\documentclass[reqno,fleqn]{amsart}
\usepackage{amssymb}
\usepackage{amsmath}
\usepackage{amsfonts}
\usepackage{fancyhdr}

\newtheorem{theorem}{Theorem}
\theoremstyle{plain}

\newtheorem{corollary}{Corollary}

\newtheorem{lemma}{Lemma}
\newtheorem{notation}{Notation}
\newtheorem{problem}{Problem}
\newtheorem{proposition}{Proposition}
\newtheorem{remark}{Remark}

\input{tcilatex}

\begin{document}
\title[$W^{2,2}$ interior convergence for some class of elliptic
anisotrpic....]{$W^{2,2}$ interior convergence for some class of elliptic
anisotropic singular pertubations problems}
\author{Chokri Ogabi}
\address{Acad\'{e}mie de Grenoble}
\email{chokri.ogabi@ac-grenoble.fr}
\date{22 Jun 2017}
\urladdr{}

\begin{abstract}
In this paper, we deal with anisotropic singular perturbations of some class
of elliptic problem. We study the asymptotic behavior of the solution in
certain second order pseudo Sobolev space.
\end{abstract}

\maketitle

\section{Description of the problem}

In this paper, we study diffusion problems when the diffusion coefficients
in certain directions are going toward zero. More precisely we are
interested in studying the asymptotic behavior of the solution in certain
second order pseudo Sobolev space. We consider the following elliptic problem

\begin{equation}
\left\{ 
\begin{array}{cc}
-\func{div}(A_{\epsilon }\nabla u_{\epsilon })=f &  \\ 
u_{\epsilon }\in W_{0}^{1,2}(\Omega )\text{ \ \ \ \ \ \ \ } & 
\end{array}%
\right.  \label{1}
\end{equation}%
where $0<\epsilon \leq 1$ and $\Omega $ is a bounded domain (i.e. open
bounded connected subset) of $%
\mathbb{R}
^{N}$ and $f\in L^{2}(\Omega ).$ We denote by $%
x=(x_{1},...,x_{N})=(X_{1},X_{2})$ the points in $%
\mathbb{R}
^{N}$ where

\begin{equation*}
X_{1}=(x_{1},...,x_{q})\text{ and }X_{1}=(x_{q+1},...,x_{N}),
\end{equation*}%
with this notation we set%
\begin{equation*}
\nabla =(\partial _{x_{1}},...,\partial _{x_{N}})^{T}=\left( 
\begin{array}{c}
\nabla _{X_{1}} \\ 
\nabla _{X_{2}}%
\end{array}%
\right) ,
\end{equation*}%
where%
\begin{equation*}
\nabla _{X_{1}}=(\partial _{x_{1}},...,\partial _{x_{q}})^{T}\text{ and }%
\nabla _{X_{2}}=(\partial _{x_{q+1}},...,\partial _{x_{N}})^{T}
\end{equation*}%
The diffusion matrix $A_{\epsilon }$ is given by 
\begin{equation*}
A_{\epsilon }=(a_{ij}^{\epsilon })=\left( 
\begin{array}{cc}
\epsilon ^{2}A_{11} & \epsilon A_{12} \\ 
\epsilon A_{21} & A_{22}%
\end{array}%
\right) \text{ with }A=(a_{ij})=\left( 
\begin{array}{cc}
A_{11} & A_{12} \\ 
A_{21} & A_{22}%
\end{array}%
\right) ,
\end{equation*}%
where $A_{11}$ and $A_{22}$ are $q\times q$ and $(N-q)\times (N-q)$
matrices. The coefficients $a_{ij}^{\epsilon }$ are given by 
\begin{equation*}
a_{ij}^{\epsilon }=\left\{ 
\begin{array}{c}
\epsilon ^{2}a_{ij}\text{ for }i,j\in \left\{ 1,..,q\right\} \text{ \ \ \ \
\ \ \ \ \ \ \ \ \ \ \ \ \ \ \ \ \ \ } \\ 
a_{ij}\text{ for }i,j\in \left\{ q+1,..,N\right\} \text{ \ \ \ \ \ \ \ \ \ \
\ \ \ \ \ \ \ \ \ } \\ 
\epsilon a_{ij}\text{ for }i\in \left\{ 1,..,q\right\} \text{, }j\in \left\{
q+1,..,N\right\} \\ 
\epsilon a_{ij}^{\epsilon }\text{ for }i\in \left\{ q+1,..,N\right\} \text{, 
}j\in \left\{ 1,..,q\right\}%
\end{array}%
\right.
\end{equation*}%
We assume that $A\in L^{\infty }(\Omega )$ and for some $\lambda >0$ we have 
\begin{equation}
A(x)\zeta \cdot \zeta \geq \lambda \left\vert \zeta \right\vert ^{2}\text{,}%
\forall \zeta \in 
\mathbb{R}
^{N},\text{ a.e }x\in \Omega .  \label{2}
\end{equation}%
Recall the Hilbert space introduced in \cite{integro} 
\begin{equation*}
V^{1,2}=\left\{ u\in L^{2}(\Omega )\text{ }\mid \nabla _{X_{2}}u\in
L^{2}(\Omega )\text{ and }u(X_{1},\cdot )\in W_{0}^{1,2}(\Omega _{X_{1}})%
\text{ a.e }X_{1}\in \Omega ^{1}\text{ }\right\} ,.
\end{equation*}%
equipped with the norm 
\begin{equation*}
\left\Vert u\right\Vert _{1,2}=\left( \left\Vert u\right\Vert _{L^{2}(\Omega
)}^{2}+\left\Vert \nabla _{X_{2}}u\right\Vert _{L^{2}(\Omega )}^{2}\right) ^{%
\frac{1}{2}}.
\end{equation*}%
Here $\Omega _{X_{1}}=\left\{ X_{2}\in 
\mathbb{R}
^{N-q}:(X_{1},X_{2})\in \Omega \right\} $ and $\Omega ^{1}=P_{1}(\Omega )$
where $P_{1}$ is the natural projector $%
\mathbb{R}
^{N}\rightarrow 
\mathbb{R}
^{q}.$

We introduce the second order local pseudo Sobolev space%
\begin{equation*}
V_{loc}^{2,2}=\left\{ u\in V^{1,2}\mid \nabla _{X_{2}}^{2}u\in
L_{loc}^{2}(\Omega )\right\} ,
\end{equation*}%
equipped with the family of norms $(\left\Vert \cdot \right\Vert
_{2,2}^{\omega })_{\omega }$ given by 
\begin{equation*}
\left\Vert u\right\Vert _{2,2}^{\omega }=\left( \left\Vert u\right\Vert
_{L^{2}(\Omega )}^{2}+\left\Vert \nabla _{X_{2}}u\right\Vert _{L^{2}(\Omega
)}^{2}+\left\Vert \nabla _{X_{2}}^{2}u\right\Vert _{L^{2}(\omega
)}^{2}\right) ^{\frac{1}{2}}\text{, }\omega \subset \subset \Omega \text{
open}
\end{equation*}%
where $\nabla _{X_{2}}^{2}u$ is the Hessian matrix of $u$ taken in the $%
X_{2} $ direction, the term $\left\Vert \nabla _{X_{2}}^{2}u\right\Vert
_{L^{2}(\omega )}^{2}$ is given by 
\begin{equation*}
\left\Vert \nabla _{X_{2}}^{2}u\right\Vert _{L^{2}(\omega
)}^{2}=\dsum\limits_{i,j=q+1}^{N}\left\Vert \partial _{ij}^{2}u\right\Vert
_{L^{2}(\omega )}^{2}.
\end{equation*}%
We can show that $V_{loc}^{2,2}$ is a Fr\'{e}chet space (i.e. locally
convex, metrizable and complete). We also define the following%
\begin{equation*}
\left\Vert \nabla _{X_{1}}^{2}u\right\Vert _{L^{2}(\omega
)}^{2}=\dsum\limits_{i,j=1}^{q}\left\Vert \partial _{ij}^{2}u\right\Vert
_{L^{2}(\omega )}^{2},
\end{equation*}%
and%
\begin{equation*}
\left\Vert \nabla _{X_{1}X_{2}}^{2}u\right\Vert _{L^{2}(\omega
)}^{2}=\dsum\limits_{i=1}^{q}\dsum\limits_{j=q+1}^{N}\left\Vert \partial
_{ij}^{2}u\right\Vert _{L^{2}(\omega )}^{2}.
\end{equation*}

As $\epsilon \rightarrow 0$, the Limit problem is given by

\begin{equation}
\left\{ 
\begin{array}{cc}
-\func{div}(A_{22}\nabla u_{0}(X_{1},\cdot )=f(X_{1},\cdot )\text{ \ \ \ \ \
\ \ \ \ \ \ \ \ } &  \\ 
u_{0}(X_{1},\cdot )\in W_{0}^{1,2}(\Omega _{X_{1}})\text{ \ \ \ \ \ \ \ \ \
\ \ \ \ \ \ \ \ \ \ \ \ \ } & \text{a.e }X_{1}\in \Omega ^{1}%
\end{array}%
\right.  \label{3}
\end{equation}

The existence and the uniqueness of the $W_{0}^{1,2}$ weak solutions to (\ref%
{1}) and (\ref{3}) follow from the Lax-Milgram theorem. In \cite{chipot} the
authors studied the relationship between $u_{\epsilon }$ and $u_{0}$ and
they proved that $u_{0}\in V^{1,2}$ and the following convergences (see
Theorem 2.1 in the above reference)%
\begin{equation}
u_{\epsilon }\rightarrow u_{0}\text{ in }V^{1,2}\text{ and }\epsilon \nabla
_{X_{1}}u_{\epsilon }\rightarrow 0\text{ in }L^{2}(\Omega ).  \label{19}
\end{equation}%
For the $L^{p}$ case we refer the reader to \cite{chok}, and \cite{integro},%
\cite{chip}, \cite{nonlinearchip} for other related problems. In this paper,
we deal with the asymptotic behavior of the second derivatives of $%
u_{\epsilon }$, in other words we show the convergence of $u_{\epsilon }$ in
the space $V_{loc}^{2,2}$ introduced previously. The arguments are based on
the Riesz-Fr\'{e}chet-Kolmogorov compacity theorem in $L^{p}$ spaces. Let us
give the main result

\begin{theorem}
Assume that $A\in L^{\infty }(\Omega )\cap C^{1}(\Omega )$ with (\ref{2}),
suppose that $f\in L^{2}(\Omega )$ then $u_{0}\in V_{loc}^{2,2}$ and $%
u_{\epsilon }\rightarrow u_{0}$ in $V_{loc}^{2,2}$, where $u_{\epsilon }\in
W_{0}^{1,2}(\Omega )\cap W_{loc}^{2,2}(\Omega )$ and $u_{0}$ are the unique
weak solutions to (\ref{1}) and (\ref{3}) respectively. In addition, the
convergences $\epsilon ^{2}\nabla _{X_{1}}^{2}u_{\epsilon }\rightarrow 0$, $%
\epsilon \nabla _{X_{1}X_{2}}^{2}u_{\epsilon }\rightarrow 0$ hold in $%
L_{loc}^{2}(\Omega ).$
\end{theorem}

\section{Some useful tools}

\begin{proposition}
The vector space $V_{loc}^{2,2}$ equipped with the family of norms $%
(\left\Vert \cdot \right\Vert _{2,2}^{\omega })_{\omega }$ is a Fr\'{e}chet
space.
\end{proposition}

\begin{proof}
Let $(\omega _{n})_{n\in 
\mathbb{N}
}$ be a countable open covering of $\Omega $ with $\omega _{n}\subset
\subset $ $\Omega $, $\omega _{n}\subset \omega _{n+1}$ for every $n\in 
\mathbb{N}
.$ The countable family $(\left\Vert \cdot \right\Vert _{2,2}^{\omega
_{n}})_{n\in 
\mathbb{N}
}$ define a base of norms for the $V^{2,2}$ topology. The general theory of
locally convex topological vector spaces shows that this topology is
metrizable, explicitly a distance $d$ which define this topology is given by
( see for instance \cite{vokhac})%
\begin{equation}
d(u,v)=\dsum\limits_{n=0}^{\infty }2^{-n}\frac{\left\Vert u-v\right\Vert
_{2,2}^{\omega _{n}}}{1+\left\Vert u-v\right\Vert _{2,2}^{\omega _{n}}}\text{%
, \ \ \ }u,v\in V_{loc}^{2,2}.  \label{20}
\end{equation}%
Let $(u_{m})$ be a Cauchy sequence in $V_{loc}^{2,2}$ then $(u_{m})$ is a
Cauchy sequence for each norm $\left\Vert \cdot \right\Vert _{2,2}^{\omega
_{n}}$, $n\in 
\mathbb{N}
.$ Whence, there exist $u,v\in L^{2}(\Omega )$ such that%
\begin{equation*}
u_{m}\rightarrow u\text{, }\nabla _{X_{2}}u_{m}\rightarrow v\text{ in }%
L^{2}(\Omega )\text{,}
\end{equation*}%
and for every $n\in 
\mathbb{N}
$ fixed there exists $w_{n}\in L^{2}(\omega _{n})$ such that 
\begin{equation*}
\nabla _{X_{2}}^{2}u_{m}\rightarrow w_{n}\text{ in }L^{2}(\omega _{n}).
\end{equation*}

The continuity of $\nabla _{X_{2}}$ and $\nabla _{X_{2}}^{2}$ on $D^{\prime
}(\Omega )$ and $D^{\prime }(\omega _{n})$ shows that $v=\nabla _{X_{2}}u$
and $\nabla _{X_{2}}^{2}u=w_{n}$ for every $n\in 
\mathbb{N}
.$ Hence $u\in V_{loc}^{2,2}$ and 
\begin{equation*}
\forall n\in 
\mathbb{N}
:\left\Vert u_{m}-u\right\Vert _{2,2}^{\omega _{n}}\rightarrow 0\text{ as }%
m\rightarrow \infty .
\end{equation*}%
Finally the normal convergence of the series (\ref{20}) implies 
\begin{equation*}
d(u_{m},u)\rightarrow 0\text{ as }m\rightarrow \infty ,
\end{equation*}%
and therefore the completion of $V_{loc}^{2,2}$ follows.
\end{proof}

\begin{remark}
Notice that a sequence $(u_{m})$ in $V_{loc}^{2,2}$ converges to $u$ with
respect to $d$ if and only if $\left\Vert u_{m}-u\right\Vert _{2,2}^{\omega
}\rightarrow 0$ as $m\rightarrow \infty ,$ for every $\omega \subset \subset 
$ $\Omega $ open.
\end{remark}

Now, let us give two useful lemmas

\begin{lemma}
Let $f\in L^{2}(%
\mathbb{R}
^{N}),$ for every $\epsilon \in (0,1]$ let $u_{\epsilon }\in W^{2,2}(%
\mathbb{R}
^{N})$ such that 
\begin{equation}
-\epsilon ^{2}\Delta _{X_{1}}u_{\epsilon }(x)-\Delta _{X_{2}}u_{\epsilon
}(x)=f(x)\text{ a.e }x\in 
\mathbb{R}
^{N}  \label{10}
\end{equation}%
then for every $\epsilon \in (0,1]$ we have the bounds%
\begin{eqnarray*}
\left\Vert \nabla _{X_{2}}^{2}u_{\epsilon }\right\Vert _{L^{2}(%
\mathbb{R}
^{N})} &\leq &\left\Vert f\right\Vert _{L^{2}(%
\mathbb{R}
^{N})}, \\
\epsilon ^{2}\left\Vert \nabla _{X_{1}}^{2}u_{\epsilon }\right\Vert _{L^{2}(%
\mathbb{R}
^{N})} &\leq &\left\Vert f\right\Vert _{L^{2}(%
\mathbb{R}
^{N})}, \\
\sqrt{2}\epsilon \left\Vert \nabla _{X_{1}X_{2}}^{2}u_{\epsilon }\right\Vert
_{L^{2}(%
\mathbb{R}
^{N})} &\leq &\left\Vert f\right\Vert _{L^{2}(%
\mathbb{R}
^{N})}.
\end{eqnarray*}
\end{lemma}

\begin{proof}
Let $\mathcal{F}$ be the Fourier transform defined on $L^{2}(%
\mathbb{R}
^{N})$ as the extension, by density, of the Fourier transform defined on the
Schwartz space $\mathcal{S}(%
\mathbb{R}
^{N})$ by%
\begin{equation*}
\mathcal{F}(u)(\xi )=(2\pi )^{-\frac{N}{2}}\int_{%
\mathbb{R}
^{N}}u(x)e^{-ix\cdot \xi }dx\text{, \ \ }u\in \mathcal{S}(%
\mathbb{R}
^{N})
\end{equation*}%
where $\cdot $ is the standard scalar product of $%
\mathbb{R}
^{N}$. Applying $\mathcal{F}$ on (\ref{10}) we obtain%
\begin{equation*}
\left( \epsilon ^{2}\dsum\limits_{i=1}^{q}\xi _{i}^{2}+\dsum_{i=q+1}^{N}\xi
_{i}^{2}\right) \mathcal{F}(u_{\epsilon })(\xi )=\mathcal{F}(f)(\xi ),
\end{equation*}%
then 
\begin{equation}
\left( \epsilon ^{4}\dsum\limits_{i,j=1}^{q}\xi _{i}^{2}\xi
_{j}^{2}+\dsum_{i,j=q+1}^{N}\xi _{i}^{2}\xi _{j}^{2}+2\epsilon
^{2}\dsum_{j=q+1}^{N}\dsum\limits_{i=1}^{q}\xi _{i}^{2}\xi _{j}^{2}\right)
\left\vert \mathcal{F}(u_{\epsilon })(\xi )\right\vert ^{2}=\left\vert 
\mathcal{F}(f)(\xi )\right\vert ^{2},  \label{4}
\end{equation}%
thus%
\begin{equation*}
\dsum_{i,j=q+1}^{N}\xi _{i}^{2}\xi _{j}^{2}\left\vert \mathcal{F}%
(u_{\epsilon })(\xi )\right\vert ^{2}\leq \left\vert \mathcal{F}(f)(\xi
)\right\vert ^{2},
\end{equation*}%
hence%
\begin{equation*}
\dsum_{i,j=q+1}^{N}\left\vert \mathcal{F}(\partial _{ij}^{2}u_{\epsilon
})(\xi )\right\vert ^{2}\leq \left\vert \mathcal{F}(f)(\xi )\right\vert ^{2},
\end{equation*}%
then 
\begin{equation*}
\dsum_{i,j=q+1}^{N}\left\Vert \mathcal{F}(\partial _{ij}^{2}u_{\epsilon
})\right\Vert _{L^{2}(%
\mathbb{R}
^{N})}^{2}\leq \left\Vert \mathcal{F}(f)\right\Vert _{L^{2}(%
\mathbb{R}
^{N})}^{2},
\end{equation*}%
and the Parseval identity gives%
\begin{equation*}
\dsum_{i,j=q+1}^{N}\left\Vert \partial _{ij}^{2}u_{\epsilon }\right\Vert
_{L^{2}(%
\mathbb{R}
^{N})}^{2}\leq \left\Vert f\right\Vert _{L^{2}(%
\mathbb{R}
^{N})}^{2}.
\end{equation*}%
Hence%
\begin{equation*}
\left\Vert \nabla _{X_{2}}^{2}u_{\epsilon }\right\Vert _{L^{2}(%
\mathbb{R}
^{N})}\leq \left\Vert f\right\Vert _{L^{2}(%
\mathbb{R}
^{N})}.
\end{equation*}

Similarly we obtain from (\ref{4}) the bounds 
\begin{equation*}
\epsilon ^{2}\left\Vert \nabla _{X_{1}}^{2}u_{\epsilon }\right\Vert _{L^{2}(%
\mathbb{R}
^{N})}\leq \left\Vert f\right\Vert _{L^{2}(%
\mathbb{R}
^{N})},
\end{equation*}%
\begin{equation*}
\sqrt{2}\epsilon \left\Vert \nabla _{X_{1}X_{2}}^{2}u_{\epsilon }\right\Vert
_{L^{2}(%
\mathbb{R}
^{N})}\leq \left\Vert f\right\Vert _{L^{2}(%
\mathbb{R}
^{N})}.
\end{equation*}
\end{proof}

\begin{notation}
For a function $u\in L^{p}(%
\mathbb{R}
^{N})$ and $h\in 
\mathbb{R}
^{N}$ we denote $\tau _{h}u(x)=u(x+h),$ $x\in 
\mathbb{R}
^{N}.$
\end{notation}

\begin{lemma}
Let $\Omega $ be an open bounded subset of $%
\mathbb{R}
^{N}$ and let $(u_{k})_{k\in 
\mathbb{N}
}$ be a converging sequence in $L^{p}(\Omega )$,$1\leq p<\infty $ and let $%
\omega \subset \subset \Omega $ open, then for every $\sigma >0$ there
exists $0<\delta <dist(\partial \Omega ,\omega )$ such that 
\begin{equation*}
\forall h\in 
\mathbb{R}
^{N},\left\vert h\right\vert \leq \delta ,\forall k\in 
\mathbb{N}
:\left\Vert \tau _{h}u_{k}-u_{k}\right\Vert _{L^{p}(\omega )}\leq \sigma
\end{equation*}%
in other words we have $\underset{h\rightarrow 0}{\lim }$ $\underset{k\in 
\mathbb{N}
}{\sup }\left\Vert \tau _{h}u_{k}-u_{k}\right\Vert _{L^{p}(\omega )}=0.$
\end{lemma}

\begin{proof}
Let $\omega \subset \subset \Omega $ open. For a function $v\in L^{p}(\Omega
),$ extend $v$ by $0$ outside of $\Omega ,$ since the translation $%
h\rightarrow \tau _{h}v$ is continuous from $%
\mathbb{R}
^{N}$ to $L^{p}(%
\mathbb{R}
^{N})$ (see for instance \cite{vokhac}) then for every $\sigma >0$ there
exists $0<\delta <dist(\partial \Omega ,\omega )$ such that 
\begin{equation}
\forall h\in 
\mathbb{R}
^{N},\left\vert h\right\vert \leq \delta :\left\Vert \tau _{h}v-v\right\Vert
_{L^{p}(\omega )}\leq \sigma .  \label{5}
\end{equation}

We denote $\lim u_{k}=u\in L^{p}(\Omega )$, and let $\sigma >0$ then (\ref{5}%
) shows that there exists $0<\delta <dist(\partial \Omega ,\omega )$ such
that%
\begin{equation*}
\forall h\in 
\mathbb{R}
^{N},\left\vert h\right\vert \leq \delta :\left\Vert \tau _{h}u-u\right\Vert
_{L^{p}(\omega )}\leq \frac{\sigma }{2}.
\end{equation*}%
By the triangular inequality and the invariance of the Lebesgue measure
under translations we have for every $k\in 
\mathbb{N}
$ and $\left\vert h\right\vert \leq \delta $ 
\begin{equation}
\left\Vert \tau _{h}u_{k}-u_{k}\right\Vert _{L^{p}(\omega )}\leq 2\left\Vert
u_{k}-u\right\Vert _{L^{p}(\Omega )}+\left\Vert \tau _{h}u-u\right\Vert
_{L^{p}(\omega )}  \label{6}
\end{equation}%
Since $u_{k}\rightarrow u$ in $L^{p}(\Omega )$ then there exists $k_{0}\in 
\mathbb{N}
$, such that 
\begin{equation*}
\forall k\geq k_{0}:\left\Vert u_{k}-u\right\Vert _{L^{p}(\Omega )}\leq 
\frac{\sigma }{4}.
\end{equation*}%
Then from (\ref{6}) we obtain 
\begin{equation}
\forall h\in 
\mathbb{R}
^{N},\left\vert h\right\vert \leq \delta ,\forall k\geq k_{0}:\left\Vert
\tau _{h}u_{k}-u_{k}\right\Vert _{L^{p}(\omega )}\leq \sigma  \label{7}
\end{equation}

Similarly (\ref{5}) shows that for every $k\in \left\{
0,1,2,...,k_{0}-1\right\} $ there exists $0<\delta _{k}<dist(\partial \Omega
,\omega )$ such that 
\begin{equation}
\forall h\in 
\mathbb{R}
^{N},\left\vert h\right\vert \leq \delta _{k}:\left\Vert \tau
_{h}u_{k}-u_{k}\right\Vert _{L^{p}(\omega )}\leq \sigma ,\text{ }k\in
\left\{ 0,1,2,...,k_{0}-1\right\}  \label{8}
\end{equation}%
Taking $\delta ^{\prime }=\underset{k\in \left\{ 0,..,k_{0}-1\right\} }{\min 
}(\delta _{k},\delta )$ and combining (\ref{7}) and (\ref{8}) we obtain 
\begin{equation*}
\forall h\in 
\mathbb{R}
^{N},\left\vert h\right\vert \leq \delta ^{\prime },\forall k\in 
\mathbb{N}
:\left\Vert \tau _{h}u_{k}-u_{k}\right\Vert _{L^{p}(\omega )}\leq \sigma .
\end{equation*}
\end{proof}

\section{\protect\bigskip The perturbed Laplace equation}

In this section we will prove \textbf{Theorem 1} for the perturbed Laplace
equation. We suppose that $A=Id$, and let $u_{\epsilon }\in
W_{0}^{1,2}(\Omega )$ be the unique solution to 
\begin{equation}
\left\{ 
\begin{array}{cc}
-\epsilon ^{2}\Delta _{X_{1}}u_{\epsilon }-\Delta _{X_{2}}u_{\epsilon }=f & 
\\ 
u_{\epsilon }\in W_{0}^{1,2}(\Omega ).\text{\ \ \ \ \ \ \ \ \ \ \ \ \ } & 
\end{array}%
\right.  \label{9}
\end{equation}%
Notice that the elliptic regularity \cite{trudinger} shows that $u_{\epsilon
}\in W_{loc}^{2,2}(\Omega )$. Now, let $(\epsilon _{k})_{k\in 
\mathbb{N}
}$ be a sequence in $(0,1]$ with $\lim \epsilon _{k}=0,$ and let $%
u_{k}=u_{\epsilon _{k}}$ be the solution of (\ref{9}) with $\epsilon $
replaced by $\epsilon _{k}$. then one can prove the following

\begin{proposition}
1) Let $\omega \subset \subset \Omega $ open then 
\begin{equation*}
\begin{array}{cc}
\underset{h\rightarrow 0}{\lim \text{ }}\underset{k\in 
\mathbb{N}
}{\sup }\left\Vert \tau _{h}\nabla _{X_{2}}^{2}u_{k}-\nabla
_{X_{2}}^{2}u_{k}\right\Vert _{L^{2}(\omega )}=0,\text{ \ \ \ \ \ \ \ \ \ \ }
& \text{ \ \ \ \ \ \ } \\ 
\underset{h\rightarrow 0}{\lim \text{ }}\underset{k\in 
\mathbb{N}
}{\sup }\left\Vert \epsilon _{k}^{2}(\tau _{h}\nabla
_{X_{1}}^{2}u_{k}-\nabla _{X_{1}}^{2}u_{k})\right\Vert _{L^{2}(\omega )}=0,%
\text{ \ \ \ \ } &  \\ 
\underset{h\rightarrow 0}{\lim \text{ }}\underset{k\in 
\mathbb{N}
}{\sup }\left\Vert \epsilon _{k}(\tau _{h}\nabla
_{X_{1}X_{2}}^{2}u_{k}-\nabla _{X_{1}X_{2}}^{2}u_{k})\right\Vert
_{L^{2}(\omega )}=0. & 
\end{array}%
\end{equation*}

2) The sequences $\left( \nabla _{X_{2}}^{2}u_{k}\right) $, $(\epsilon
_{k}^{2}\nabla _{X_{1}}^{2}u_{k})$, $(\epsilon _{k}\nabla
_{X_{1}X_{2}}^{2}u_{k})$ are bounded in $L_{loc}^{2}(\Omega )$ i.e. for
every $\omega \subset \subset \Omega $ open there exists $M\geq 0$ such that%
\begin{equation*}
\underset{k\in 
\mathbb{N}
}{\sup }\left\Vert \epsilon _{k}^{2}\nabla _{X_{1}}^{2}u_{k}\right\Vert
_{L^{2}(\omega )},\underset{k\in 
\mathbb{N}
}{\sup }\left\Vert \nabla _{X_{2}}^{2}u_{k}\right\Vert _{L^{2}(\omega )},%
\underset{k\in 
\mathbb{N}
}{\sup }\left\Vert \epsilon _{k}\nabla _{X_{1}X_{2}}^{2}u_{k}\right\Vert
_{L^{2}(\omega )}\leq M.
\end{equation*}
\end{proposition}

\begin{proof}
1) Let $\omega \subset \subset \Omega $ open$,$ then one can choose $\omega
^{\prime }$ open such that $\omega \subset \subset \omega ^{\prime }\subset
\subset \Omega ,$ let $\rho \in \mathcal{D}(%
\mathbb{R}
^{N})$ with $\rho =1$ on $\omega $, $0\leq \rho \leq 1$ and $Supp(\rho
)\subset \omega ^{\prime }$. Let $0<h<dist(\omega ^{\prime },\partial \Omega
),$ to make the notations less heavy we set $U_{k}^{h}=\tau _{h}u_{k}-u_{k}$%
, then $U_{k}^{h}\in W^{2,2}(\omega ^{\prime }).$ Notice that translation
and derivation commute then we have%
\begin{equation*}
-\epsilon _{k}^{2}\Delta _{X_{1}}U_{k}^{h}(x)-\Delta
_{X_{2}}U_{k}^{h}(x)=F^{h}(x)\text{, \ a.e }x\in \omega ^{\prime }\text{,}
\end{equation*}%
with $F^{h}=\tau _{h}f-f$.

We set $\mathcal{W}_{k}^{h}=\rho U_{k}^{h}$ then we get 
\begin{multline*}
-\epsilon _{k}^{2}\Delta _{X_{1}}\mathcal{W}_{k}^{h}(x)-\Delta _{X_{2}}%
\mathcal{W}_{k}^{h}(x)=\rho (x)F^{h}(x)-2\epsilon _{k}^{2}\nabla
_{X_{1}}\rho (x)\cdot \nabla _{X_{1}}U_{k}^{h}(x) \\
-2\nabla _{X_{2}}\rho (x)\cdot \nabla
_{X_{2}}U_{k}^{h}(x)-U_{k}^{h}(x)(\epsilon _{k}^{2}\Delta _{X_{1}}\rho
(x)-\Delta _{X_{2}}\rho (x)),
\end{multline*}%
for a.e $x\in \omega ^{\prime }$.

Since $U_{k}^{h}\in W^{2,2}(\omega ^{\prime })$ then $\mathcal{W}_{k}^{h}\in
W_{0}^{2,2}(\omega ^{\prime })$, so we can extend $\mathcal{W}_{k}^{h}$ by $%
0 $ outside of $\omega ^{\prime }$ then $\mathcal{W}_{k}^{h}\in W^{2}(%
\mathbb{R}
^{N}).$ The right hand side of the above equality is extended by $0$ outside
of $\omega ^{\prime }$, hence the equation is satisfied in the whole space,
and thus by \textbf{Lemma 1} we get 
\begin{multline*}
\left\Vert \nabla _{X_{2}}^{2}\mathcal{W}_{k}^{h}\right\Vert _{L^{2}(%
\mathbb{R}
^{N})}\leq \left\Vert \rho F^{h}\right\Vert _{L^{2}(%
\mathbb{R}
^{N})}+2\epsilon _{k}^{2}\left\Vert \nabla _{X_{1}}\rho \cdot \nabla
_{X_{1}}U_{k}^{h}\right\Vert _{L^{2}(%
\mathbb{R}
^{N})} \\
+2\left\Vert \nabla _{X_{2}}\rho \cdot \nabla _{X_{2}}U_{k}^{h}\right\Vert
_{L^{2}(%
\mathbb{R}
^{N})}+\left\Vert U_{k}^{h}(\epsilon _{k}^{2}\Delta _{X_{1}}\rho -\Delta
_{X_{2}}\rho )\right\Vert _{L^{2}(%
\mathbb{R}
^{N})}.
\end{multline*}%
Then 
\begin{multline*}
\left\Vert \nabla _{X_{2}}^{2}U_{k}^{h}\right\Vert _{L^{2}(\omega )}\leq
\left\Vert F^{h}\right\Vert _{L^{2}(\omega ^{\prime })}+2\epsilon
_{k}\left\Vert \nabla _{X_{1}}\rho \right\Vert _{\infty }\left\Vert \epsilon
_{k}\nabla _{X_{1}}U_{k}^{h}\right\Vert _{L^{2}(\omega ^{\prime })} \\
+2\left\Vert \nabla _{X_{2}}\rho \right\Vert _{\infty }\left\Vert \nabla
_{X_{2}}U_{k}^{h}\right\Vert _{L^{2}(\omega ^{\prime })}+\left\Vert
(\epsilon _{k}^{2}\Delta _{X_{1}}\rho -\Delta _{X_{2}}\rho )\right\Vert
_{\infty }\left\Vert U_{k}^{h}\right\Vert _{L^{2}(\omega ^{\prime })}.
\end{multline*}%
Notice that by (\ref{19}) we have $u_{k}\rightarrow u$ in $V^{1,2}$ and $%
\epsilon _{k}\nabla _{X_{1}}u_{k}\rightarrow 0$ in $L^{2}(\Omega )$, then by 
\textbf{Lemma 2} we deduce 
\begin{equation*}
\underset{h\rightarrow 0}{\lim \text{ }}\underset{k\in 
\mathbb{N}
}{\sup }\left\Vert \epsilon _{k}\nabla _{X_{1}}U_{k}^{h}\right\Vert
_{L^{2}(\omega ^{\prime })}=\underset{h\rightarrow 0}{\lim \text{ }}\underset%
{k\in 
\mathbb{N}
}{\sup }\left\Vert \epsilon _{k}(\tau _{h}\nabla _{X_{1}}u_{k}-\nabla
_{X_{1}}u_{k})\right\Vert _{L^{2}(\omega ^{\prime })}=0\text{,}
\end{equation*}%
and similarly we obtain%
\begin{eqnarray*}
\underset{h\rightarrow 0}{\lim \text{ }}\underset{k\in 
\mathbb{N}
}{\sup }\left\Vert \nabla _{X_{2}}U_{k}^{h}\right\Vert _{L^{2}(\omega
^{\prime })} &=&0\text{, }\underset{h\rightarrow 0}{\lim \text{ }}\underset{%
k\in 
\mathbb{N}
}{\sup }\left\Vert F^{h}\right\Vert _{L^{2}(\omega ^{\prime })}=0\text{,} \\
\underset{h\rightarrow 0}{\lim \text{ }}\underset{k\in 
\mathbb{N}
}{\sup }\left\Vert U_{k}^{h}\right\Vert _{L^{2}(\omega ^{\prime })} &=&0.
\end{eqnarray*}%
and hence 
\begin{equation*}
\underset{h\rightarrow 0}{\lim \text{ }}\underset{k\in 
\mathbb{N}
}{\sup }\left\Vert \tau _{h}\nabla _{X_{2}}^{2}u_{k}-\nabla
_{X_{2}}^{2}u_{k}\right\Vert _{L^{2}(\omega )}=\underset{h\rightarrow 0}{%
\lim \text{ }}\underset{k\in 
\mathbb{N}
}{\sup }\left\Vert \nabla _{X_{2}}^{2}U_{k}^{h}\right\Vert _{L^{2}(\omega
)}=0.
\end{equation*}%
Similarly we obtain%
\begin{equation*}
\underset{h\rightarrow 0}{\lim \text{ }}\underset{k\in 
\mathbb{N}
}{\sup }\left\Vert \epsilon _{k}^{2}(\tau _{h}\nabla
_{X_{1}}^{2}u_{k}-\nabla _{X_{1}}^{2}u_{k})\right\Vert _{L^{2}(\omega )}=0,
\end{equation*}%
and%
\begin{equation*}
\underset{h\rightarrow 0}{\lim \text{ }}\underset{k\in 
\mathbb{N}
}{\sup }\left\Vert \epsilon _{k}(\tau _{h}\nabla
_{X_{1}X_{2}}^{2}u_{k}-\nabla _{X_{1}X_{2}}^{2}u_{k})\right\Vert
_{L^{2}(\omega )}=0.
\end{equation*}%
2) Following the same arguments, we get the estimation 
\begin{multline*}
\epsilon _{k}^{2}\left\Vert \nabla _{X_{1}}^{2}u_{k}\right\Vert
_{L^{2}(\omega )}+\left\Vert \nabla _{X_{2}}^{2}u_{k}\right\Vert
_{L^{2}(\omega )}+\sqrt{2}\epsilon _{k}\left\Vert \nabla
_{X_{1}X_{2}}^{2}u_{k}\right\Vert _{L^{2}(\omega )}\leq \\
3\left\Vert f\right\Vert _{L^{2}(\omega ^{\prime })}+6\epsilon
_{k}\left\Vert \nabla _{X_{1}}\rho \right\Vert _{\infty }\left\Vert \epsilon
_{k}\nabla _{X_{1}}u_{k}\right\Vert _{L^{2}(\omega ^{\prime })} \\
+6\left\Vert \nabla _{X_{2}}\rho \right\Vert _{\infty }\left\Vert \nabla
_{X_{2}}u_{k}\right\Vert _{L^{2}(\omega ^{\prime })}+3\left\Vert (\epsilon
_{k}^{2}\Delta _{X_{1}}\rho -\Delta _{X_{2}}\rho )\right\Vert _{\infty
}\left\Vert u_{k}\right\Vert _{L^{2}(\omega ^{\prime })}.
\end{multline*}%
The convergences $u_{k}\rightarrow u$ in $V^{1,2}$, $\epsilon _{k}\nabla
_{X_{1}}u_{k}\rightarrow 0$ in $L^{2}(\Omega )$ and boundedness of $\rho $
and its derivatives show that the right hand side of the above inequality is
uniformly bounded in $k$, i.e. for some $M\geq 0$ independent of $k$ we have 
\begin{equation*}
\epsilon _{k}^{2}\left\Vert \nabla _{X_{1}}^{2}u_{k}\right\Vert
_{L^{2}(\omega )}+\left\Vert \nabla _{X_{2}}^{2}u_{k}\right\Vert
_{L^{2}(\omega )}+\sqrt{2}\epsilon _{k}\left\Vert \nabla
_{X_{1}X_{2}}^{2}u_{k}\right\Vert _{L^{2}(\omega )}\leq M,\text{ }\forall
k\in 
\mathbb{N}
,
\end{equation*}%
and therefore, the sequences $\left( \nabla _{X_{2}}^{2}u_{k}\right) $, $%
(\epsilon _{k}^{2}\nabla _{X_{1}}^{2}u_{k})$, $(\epsilon _{k}\nabla
_{X_{1}X_{2}}^{2}u_{k})$ are bounded in $L_{loc}^{2}(\Omega ).$
\end{proof}

Now, we are ready to prove the following

\begin{theorem}
Let $u_{\epsilon }\in W_{0}^{1,2}(\Omega )\cap W_{loc}^{2,2}(\Omega )$ be
the solution of (\ref{9}) then $u_{\epsilon }\rightarrow u_{0}$ strongly in $%
V_{loc}^{2,2}$ where $u_{0}\in V_{loc}^{2,2}$ is the solution of the limit
problem. In addition, we have%
\begin{equation*}
\epsilon ^{2}\nabla _{X_{1}}^{2}u_{\epsilon }\rightarrow 0\text{ and }%
\epsilon \nabla _{X_{1}X_{2}}^{2}u_{\epsilon }\rightarrow 0,\text{ strongly
in }L_{loc}^{2}(\Omega ).
\end{equation*}
\end{theorem}

\begin{proof}
Let $u_{0}\in V^{1,2}$ be the solution of the limit problem and let $%
(u_{k})_{k\in 
\mathbb{N}
},$ $u_{k}=u_{\epsilon _{k}}\in W_{0}^{1,2}(\Omega )\cap
W_{loc}^{2,2}(\Omega )$ be a sequence of solutions to (\ref{9}) with $%
\epsilon $ replaced by $\epsilon _{k}$. Then \textbf{Proposition 2} shows
that the hypothesis of the Riesz-Fr\'{e}chet-Kolmogorov theorem are
fulfilled (For the statement of the theorem, see for instance \cite{ellipt}%
). Whence, it follows that $\left\{ \nabla _{X_{2}}^{2}u_{k}\right\} _{k\in 
\mathbb{N}
}$ is relatively compact in $L^{2}(\omega )$ for every $\omega \subset
\subset \Omega $ open. Now, for $\omega \subset \subset \Omega $ fixed there
exists $u_{0}^{\omega }\in L^{2}(\omega )$ and a subsequence still labeled $%
(\nabla _{X_{2}}^{2}u_{k})_{k\in 
\mathbb{N}
}$ such that $\nabla _{X_{2}}^{2}u_{k}\rightarrow u_{0}^{\omega }$ in $%
L^{2}(\omega )$ strongly. Since $u_{k}\rightarrow u_{0}$ in $L^{2}(\omega )$
and the second order differential operators $\partial _{ij}^{2}$ are
continuous on $\mathcal{D}^{\prime }(\omega )$ then $u_{0}^{\omega }=\nabla
_{X_{2}}u_{0}$ on $\omega .$ Whence, since $\omega $ is arbitrary we get $%
\nabla _{X_{2}}^{2}u_{0}\in L_{loc}^{2}(\Omega )$, i.e. $u_{0}\in
V_{loc}^{2,2}.$

Now, Let $(\omega _{n})$ be a countable covering of $\Omega $ with $\omega
_{n}\subset \subset \Omega ,$ $\omega _{n}\subset \omega _{n+1}$,$\forall
n\in 
\mathbb{N}
.$ Then by the diagonal process one can construct a subsequence still
labeled $(u_{k})$ such that 
\begin{equation*}
\nabla _{X_{2}}^{2}u_{k}\rightarrow \nabla _{X_{2}}^{2}u_{0}\text{ in }%
L_{loc}^{2}(\Omega )\text{ strongly.}
\end{equation*}%
Combining this with the convergence $u_{k}\rightarrow u_{0}$ of (\ref{19})
we get%
\begin{equation*}
u_{k}\rightarrow u_{0}\text{ strongly in }V_{loc}^{2,2},\text{ i.e. }%
d(u_{k},u_{0})\rightarrow 0\text{ as }k\rightarrow \infty ,
\end{equation*}%
where $d$ is the distance of the Fr\'{e}chet space $V_{loc}^{2,2}.$

To prove the convergence of the whole sequence $(u_{\epsilon })_{0<\epsilon
\leq 1}$ we can reason by contradiction. Suppose that there exists $\delta >0
$ and a subsequence $(u_{k})$ such that $d(u_{k},u_{0})>\delta $. It follows
by the first part of this proof that there exists a subsequence still
labeled $(u_{k})$ such that $d(u_{k},u_{0})\rightarrow 0$, which is a
contradiction..

By using the same arguments we can show easily ( see the end of subsection
4.1) that 
\begin{equation*}
\epsilon ^{2}\nabla _{X_{1}}^{2}u_{\epsilon }\rightarrow 0\text{ and }%
\epsilon \nabla _{X_{1}X_{2}}^{2}u_{\epsilon }\rightarrow 0\text{ strongly
in }L_{loc}^{2}(\Omega ).
\end{equation*}
\end{proof}

\section{\protect\bigskip General elliptic problems}

\subsection{Proof of the main theorem}

In this subsection we shall prove \textbf{Theorem 1}. Firstly, we suppose
that the coefficients of $A$ are constants then we have the following

\begin{proposition}
Suppose that the coefficients of $A$ are constants and assume (\ref{2})$,$
let $(u_{\epsilon })_{0<\epsilon \leq 1}$ be a sequence in $W^{2,2}(%
\mathbb{R}
^{N})$ such that $-\dsum\limits_{i,j}a_{ij}^{\epsilon }\partial
_{ij}^{2}u_{\epsilon }=f$, with $f\in L^{2}(%
\mathbb{R}
^{N})$ then we have for every $\epsilon \in (0,1]:$%
\begin{eqnarray*}
\lambda \left\Vert \nabla _{X_{2}}^{2}u_{\epsilon }\right\Vert _{L^{2}(%
\mathbb{R}
^{N})} &\leq &\left\Vert f\right\Vert _{L^{2}(%
\mathbb{R}
^{N})}, \\
\lambda \epsilon ^{2}\left\Vert \nabla _{X_{1}}^{2}u_{\epsilon }\right\Vert
_{L^{2}(%
\mathbb{R}
^{N})} &\leq &\left\Vert f\right\Vert _{L^{2}(%
\mathbb{R}
^{N})}, \\
\sqrt{2}\lambda \epsilon \left\Vert \nabla _{X_{1}X_{2}}^{2}u\right\Vert
_{L^{2}(%
\mathbb{R}
^{N})} &\leq &\left\Vert f\right\Vert _{L^{2}(%
\mathbb{R}
^{N})}.
\end{eqnarray*}
\end{proposition}

\begin{proof}
As in proof of \textbf{Lemma 1,} we use the Fourier transform and we obtain%
\begin{equation*}
\left( \dsum_{i,j}a_{ij}^{\epsilon }\xi _{i}\xi _{j}\right) \mathcal{F}%
(u_{\epsilon })(\xi )=\mathcal{F}(f)(\xi ),\text{ }\xi \in 
\mathbb{R}
^{N}.
\end{equation*}%
From the ellipticity assumption (\ref{2}) we deduce%
\begin{equation*}
\lambda ^{2}\left( \epsilon ^{2}\dsum_{i=}^{q}\xi
_{i}^{2}+\dsum_{i=q+1}^{N}\xi _{i}^{2}\right) ^{2}\left\vert \mathcal{F}%
(u_{\epsilon })(\xi )\right\vert ^{2}\leq \left\vert \mathcal{F}(f)(\xi
)\right\vert ^{2}.
\end{equation*}%
Thus, similarly we obtain the desired bounds.
\end{proof}

Now, suppose that $A\in L^{\infty }(\Omega )\cap C^{1}(\Omega )$ and assume (%
\ref{2}), and let $u_{\epsilon }\in W_{0}^{1,2}(\Omega )$ be the unique weak
solution to (\ref{1}), then it follows by the elliptic regularity that $%
u_{\epsilon }\in W_{loc}^{2,2}(\Omega )$. We denote $u_{k}=u_{\epsilon _{k}}$%
the solution to (\ref{1}) where $(\epsilon _{k})$ is a sequence in $(0,1]$
such that, $\epsilon _{k}\rightarrow 0$ as $k\rightarrow \infty .$

Under the above assumption we can prove the following

\begin{proposition}
\bigskip Let $z_{0}\in \Omega $ fixed then there exists $\omega _{0}\subset
\subset \Omega $ open with $z_{0}\in \omega _{0}$ such that the sequences $%
\left( \nabla _{X_{2}}^{2}u_{k}\right) $, $\left( \nabla
_{X_{1}}^{2}u_{k}\right) $ and $\left( \nabla _{X_{1}X_{2}}^{2}u_{k}\right) $
are bounded in $L^{2}(\omega _{0}).$
\end{proposition}

\begin{proof}
Since $u_{k}\in W_{0}^{1,2}(\Omega )\cap W_{loc}^{2,2}(\Omega )$ and $A\in
C^{1}(\Omega )$ then $u_{k}$ satisfies 
\begin{equation}
-\dsum\limits_{i,j}a_{ij}^{k}(x)\partial
_{ij}^{2}u_{k}(x)-\dsum\limits_{i,j}\partial _{i}a_{ij}^{k}(x)\partial
_{j}u_{k}(x)=f(x),\text{ for a.e }x\in \Omega  \label{14}
\end{equation}%
where we have set $a_{ij}^{k}=a_{ij}^{\epsilon _{k}}$.

Let $z_{0}\in \Omega $ fixed, and let $\theta >0$ such that 
\begin{equation}
\min \left\{ \left[ \lambda -3\theta (N-q)\right] ,\left[ \lambda -3\theta q%
\right] ,\left[ \sqrt{2}\lambda -6(N-q)q\theta \right] \right\} \geq \frac{%
\lambda }{2}\text{.}  \label{13}
\end{equation}%
By using the continuity of the $a_{ij}$ one can choose $\omega _{1}\subset
\subset \Omega $, $z_{0}\in \omega _{1}$ such that 
\begin{equation}
\underset{i,j}{\max }\underset{x\in \omega _{1}}{\sup }\left\vert
a_{ij}(x)-a_{ij}(z_{0})\right\vert \leq \theta \text{ }  \label{11}
\end{equation}%
Let $\omega _{0}\subset \subset \omega _{1}$ open with $z_{0}\in \omega _{0}$
and let $\rho \in \mathcal{D}(%
\mathbb{R}
^{N})$ such $\rho =1$ on $\omega _{0}$, $0\leq \rho \leq 1$ and $Supp(\rho
)\subset \omega _{1}$. We set $U_{k}=\rho u_{k}$, and we extend it by $0$ on
the outside of $\omega _{1}$ then $U_{k}\in W^{2,2}(%
\mathbb{R}
^{N})$. Therefore from (\ref{14}) we obtain%
\begin{equation*}
-\dsum\limits_{i,j}a_{ij}^{k}(z_{0})\partial
_{ij}^{2}U_{k}(x)=\dsum\limits_{i,j}(a_{ij}^{k}(x)-a_{ij}^{k}(z_{0}))%
\partial _{ij}^{2}U_{k}(x)+g_{k}(x)\text{, \ for a.e }x\in 
\mathbb{R}
^{N},
\end{equation*}%
where $g_{k}$ is given by%
\begin{eqnarray}
g_{k}(x) &=&\rho (x)f(x)+\rho (x)\dsum\limits_{i,j}\partial
_{i}a_{ij}^{k}(x)\partial _{j}u_{k}(x)  \label{12} \\
&&-u_{k}(x)\dsum\limits_{i,j}a_{ij}^{k}(x)\partial _{ij}^{2}\rho
(x)-\dsum\limits_{i,j}a_{ij}^{k}(x)\partial _{i}\rho (x)\partial
_{j}u_{k}(x)-\dsum\limits_{i,j}a_{ij}^{k}(x)\partial _{j}\rho (x)\partial
_{i}u_{k}(x),  \notag
\end{eqnarray}%
and we have extended $g_{k}$ by $0$ outside of $\omega _{1}.$

Now, applying \textbf{Proposition 3} to the above differential equality we
get 
\begin{multline*}
\lambda \left\Vert \nabla _{X_{2}}^{2}U_{k}\right\Vert _{L^{2}(\omega
_{1})}+\lambda \epsilon _{k}^{2}\left\Vert \nabla
_{X_{1}}^{2}U_{k}\right\Vert _{L^{2}(\omega _{1})}+\sqrt{2}\lambda \epsilon
_{k}\left\Vert \nabla _{X_{1}X_{2}}^{2}U_{k}\right\Vert _{L^{2}(\omega _{1})}
\\
\leq 3\left\Vert \dsum\limits_{i,j}(a_{ij}^{k}-a_{ij}^{k}(z_{0}))\partial
_{ij}^{2}U_{k}\right\Vert _{L^{2}(\omega _{1})}+3\left\Vert g\right\Vert
_{L^{2}(\omega _{1})}
\end{multline*}%
Whence, by using (\ref{11}) we get%
\begin{multline*}
\lambda \left\Vert \nabla _{X_{2}}^{2}U_{k}\right\Vert _{L^{2}(\omega
_{1})}+\lambda \epsilon _{k}^{2}\left\Vert \nabla
_{X_{1}}^{2}U_{k}\right\Vert _{L^{2}(\omega _{1})}+\sqrt{2}\lambda \epsilon
_{k}\left\Vert \nabla _{X_{1}X_{2}}^{2}U_{k}\right\Vert _{L^{2}(\omega _{1})}
\\
\leq 3\theta \epsilon _{k}^{2}\dsum\limits_{i,j=1}^{q}\left\Vert \partial
_{ij}^{2}U_{k}\right\Vert _{L^{2}(\omega _{1})}+3\theta
\dsum\limits_{i,j=q+1}^{N}\left\Vert \partial _{ij}^{2}U_{k}\right\Vert
_{L^{2}(\omega _{1})} \\
+6\theta \epsilon
_{k}\dsum\limits_{i=1}^{q}\dsum\limits_{j=q+1}^{N}\left\Vert \partial
_{ij}^{2}U_{k}\right\Vert _{L^{2}(\omega _{1})}+3\left\Vert g\right\Vert
_{L^{2}(\omega _{1})},
\end{multline*}%
and thus by the discrete Cauchy-Schwarz inequality we deduce%
\begin{multline*}
\lambda \left\Vert \nabla _{X_{2}}^{2}U_{k}\right\Vert _{L^{2}(\omega
_{1})}+\lambda \epsilon _{k}^{2}\left\Vert \nabla
_{X_{1}}^{2}U_{k}\right\Vert _{L^{2}(\omega _{1})}+\sqrt{2}\lambda \epsilon
_{k}\left\Vert \nabla _{X_{1}X_{2}}^{2}U_{k}\right\Vert _{L^{2}(\omega _{1})}
\\
\leq 3\theta (N-q)\left\Vert \nabla _{X_{2}}^{2}U_{k}\right\Vert
_{L^{2}(\omega _{1})}+\epsilon _{k}^{2}3\theta q\left\Vert \nabla
_{X_{1}}^{2}U_{k}\right\Vert _{L^{2}(\omega _{1})} \\
+\epsilon _{k}6(N-q)q\theta \left\Vert \nabla
_{X_{1}X_{2}}^{2}U_{k}\right\Vert _{L^{2}(\omega _{1})}^{2}+3\left\Vert
g\right\Vert _{L^{2}(\omega _{1})},
\end{multline*}%
and thus%
\begin{multline*}
\left[ \lambda -3\theta (N-q)\right] \left\Vert \nabla
_{X_{2}}^{2}U_{k}\right\Vert _{L^{2}(\omega _{1})}+\epsilon _{k}^{2}\left[
\lambda -3\theta q\right] \left\Vert \nabla _{X_{1}}^{2}U_{k}\right\Vert
_{L^{2}(\omega _{1})}+ \\
\epsilon _{k}\left[ \sqrt{2}\lambda -6(N-q)q\theta \right] \left\Vert \nabla
_{X_{1}X_{2}}^{2}U_{k}\right\Vert _{L^{2}(\omega _{1})}\leq 3\left\Vert
g_{k}\right\Vert _{L^{2}(\omega _{1})}.
\end{multline*}%
Hence, by (\ref{13}) we get%
\begin{equation*}
\left\Vert \nabla _{X_{2}}^{2}u_{k}\right\Vert _{L^{2}(\omega
_{0})}+\epsilon _{k}^{2}\left\Vert \nabla _{X_{1}}^{2}u_{k}\right\Vert
_{L^{2}(\omega _{0})}+\epsilon _{k}\left\Vert \nabla
_{X_{1}X_{2}}^{2}u_{k}\right\Vert _{L^{2}(\omega _{0})}\leq \frac{6}{\lambda 
}\left\Vert g_{k}\right\Vert _{L^{2}(\omega _{1})}.
\end{equation*}%
To complete the proof, we will show the boundedness of $(g_{k})$ in $%
L^{2}(\omega _{1})$. Indeed, $\rho $ and its derivatives, $a_{ij}$ and their
first derivatives are bounded on $\omega _{1}$, moreover (\ref{19}) shows
that the sequences $\left( \epsilon _{k}\nabla _{X_{1}}u_{k}\right) $, $%
\left( \nabla _{X_{2}}u_{k}\right) $ and $\left( u_{k}\right) $ are bounded
in $L^{2}(\Omega ),$ and therefore from (\ref{12}) the boundedness of $%
(g_{k})$ in $L^{2}(\omega _{1})$ follows.
\end{proof}

\begin{corollary}
The sequences $\left( \nabla _{X_{2}}^{2}u_{k}\right) $, $\left( \epsilon
_{k}^{2}\nabla _{X_{1}}^{2}u_{k}\right) $, $\left( \epsilon _{k}\nabla
_{X_{1}X_{2}}^{2}u_{k}\right) $ are bounded in $L_{loc}^{2}(\Omega ).$
\end{corollary}

\begin{proof}
Let $\omega \subset \subset \Omega $ open, for every $z\in \bar{\omega}$
there exists $\omega _{z}\subset \subset \Omega $, $z\in \omega _{z}$ which
satisfies the affirmations of \textbf{Proposition 4 }in $L^{2}(\omega _{z})$%
. By using the compacity of $\bar{\omega}$, one can extract a finite cover $%
(\omega _{z_{i}})$, and hence the sequences $\left( \nabla
_{X_{2}}^{2}u_{k}\right) $, $\left( \epsilon _{k}^{2}\nabla
_{X_{1}}^{2}u_{k}\right) $, $\left( \epsilon _{k}\nabla
_{X_{1}X_{2}}^{2}u_{k}\right) $ are bounded in $L^{2}(\omega )$.
\end{proof}

\begin{proposition}
\bigskip Let $z_{0}\in \Omega $ then there exists $\omega _{0}\subset
\subset \Omega $, $z_{0}\in \omega _{0}$ such that%
\begin{equation*}
\begin{array}{cc}
\underset{h\rightarrow 0}{\lim }\underset{k\in 
\mathbb{N}
}{\text{ }\sup }\left\Vert \tau _{h}\nabla _{X_{2}}^{2}u_{k}-\nabla
_{X_{2}}^{2}u_{k}\right\Vert _{L^{p}(\omega _{0})}=0\text{, \ \ \ \ \ \ \ \
\ \ \ \ } &  \\ 
\underset{h\rightarrow 0}{\lim }\underset{k\in 
\mathbb{N}
}{\text{ }\sup }\left\Vert \epsilon _{k}^{2}(\tau _{h}\nabla
_{X_{1}}^{2}u_{k}-\nabla _{X_{1}}^{2}u_{k})\right\Vert _{L^{p}(\omega
_{0})}=0\text{, \ \ \ \ \ } &  \\ 
\underset{h\rightarrow 0}{\lim }\underset{k\in 
\mathbb{N}
}{\text{ }\sup }\left\Vert \epsilon _{k}(\tau _{h}\nabla
_{X_{1}X_{2}}^{2}u_{k}-\nabla _{X_{1}X_{2}}^{2}u_{k})\right\Vert
_{L^{p}(\omega _{0})}=0\text{.} & 
\end{array}%
\end{equation*}
\end{proposition}

\begin{proof}
Let $z_{0}\in \Omega $ fixed and let $\theta >0$ then using the continuity
of the $a_{ij}$ one can choose $\omega _{1}\subset \subset \Omega $, $%
z_{0}\in \omega _{1}$ such that we have (\ref{11}) with $\theta $ is chosen
as in (\ref{13}). Let $\omega _{0}\subset \subset \omega _{1},$ with $%
z_{0}\in \omega _{0},$ and let $\rho \in \mathcal{D}(%
\mathbb{R}
^{N})$ with $\rho =1$ on $\omega _{0}$, $0\leq \rho \leq 1,$ and $Supp(\rho
)\subset \omega _{1}$. Let $0<h<dist(\omega _{1},\partial \Omega ),$ we set $%
\mathcal{W}_{k}^{h}=\rho U_{h}^{k}$, with $U_{k}^{h}=(\tau _{h}u_{k}-u_{k})$
and extend it by $0$ on the outside of $\omega _{1}$ then $\mathcal{W}%
_{k}^{h}\in W^{2,2}(%
\mathbb{R}
^{N})$, therefore using (\ref{14}) we have$:$ 
\begin{equation*}
-\dsum\limits_{i,j}a_{ij}^{k}(z_{0})\partial _{ij}^{2}\mathcal{W}%
_{k}^{h}(x)=\dsum\limits_{i,j}(a_{ij}^{k}(x)-a_{ij}^{k}(z_{0}))\partial
_{ij}^{2}\mathcal{W}_{k}^{h}(x)+G_{k}^{h}(x),\text{ a.e }x\in 
\mathbb{R}
^{N}
\end{equation*}%
where%
\begin{eqnarray}
-G_{k}^{h}(x) &=&U_{k}^{h}\dsum\limits_{i,j}a_{ij}^{k}(x)\partial
_{ij}^{2}\rho +\dsum\limits_{i,j}a_{ij}^{k}(x)\partial _{i}\rho \partial
_{j}U_{k}^{h}+\dsum\limits_{i,j}a_{ij}^{k}(x)\partial _{j}\rho \partial
_{i}U_{k}^{h}  \label{15} \\
&&+\rho \dsum\limits_{i,j}\left( a_{ij}^{k}(x)-\tau _{h}a_{ij}^{k}(x)\right)
\tau _{h}\partial _{ij}^{2}u_{k}(x)+\rho (x)\left( f(x)-\tau _{h}f(x)\right)
\notag \\
&&+\rho \dsum\limits_{i,j}\left[ \partial _{i}a_{ij}^{k}(x)\partial
_{j}u_{k}(x)-\partial _{i}\tau _{h}a_{ij}^{k}(x)\partial _{j}\tau
_{h}u_{k}(x)\right] ,  \notag
\end{eqnarray}%
and $G_{k}^{h}$ is extended by $0$ outside of $\omega _{1}.$

Then, as in proof of \textbf{Proposition 4}, we obtain 
\begin{multline*}
\left\Vert \tau _{h}\nabla _{X_{2}}^{2}u_{k}-\nabla
_{X_{2}}^{2}u_{k}\right\Vert _{L^{2}(\omega _{0})}+\epsilon
_{k}^{2}\left\Vert \tau _{h}\nabla _{X_{1}}^{2}u_{k}-\nabla
_{X_{1}}^{2}u_{k}\right\Vert _{L^{2}(\omega _{0})} \\
+\epsilon _{k}\left\Vert \tau _{h}\nabla _{X_{1}X_{2}}^{2}u_{k}-\nabla
_{X_{1}X_{2}}^{2}u_{k}\right\Vert _{L^{2}(\omega _{0})}\leq \frac{6}{\lambda 
}\left\Vert G_{k}^{h}\right\Vert _{L^{2}(\omega _{1})}.
\end{multline*}

To complete the proof, we have to show that $\underset{h\rightarrow 0}{\lim }%
\underset{k\in 
\mathbb{N}
}{\text{ }\sup }\left\Vert G_{k}^{h}\right\Vert _{L^{2}(\omega _{1})}=0$.

Using the boundedness of the $a_{ij}$ and the boundedness of $\rho $ and its
derivatives on $\omega _{1}$ we get from (\ref{15})%
\begin{eqnarray}
\left\Vert G_{k}^{h}\right\Vert _{L^{2}(\omega _{1})} &\leq &M\left\Vert
U_{k}^{h}\right\Vert _{L^{2}(\omega _{1})}+M\epsilon _{k}\left\Vert \nabla
_{X_{1}}U_{k}^{h}\right\Vert _{L^{2}(\omega _{1})}  \label{16} \\
&&+M\left\Vert \nabla _{X_{2}}U_{k}^{h}\right\Vert _{L^{2}(\omega
_{1})}+\left\Vert \tau _{h}f-f\right\Vert _{L^{2}(\omega _{1})}  \notag \\
&&+\dsum\limits_{i,j}\left\Vert \left( a_{ij}^{k}-\tau _{h}a_{ij}^{k}\right)
\tau _{h}\partial _{ij}^{2}u_{k}\right\Vert _{L^{2}(\omega _{1})}  \notag \\
&&+\dsum\limits_{i,j}\left\Vert \partial _{i}a_{ij}^{k}\partial
_{j}u_{k}-\tau _{h}\partial _{i}a_{ij}^{k}\tau _{h}\partial
_{j}u_{k}\right\Vert _{L^{2}(\omega _{1})},  \notag
\end{eqnarray}%
where $M\geq 0$ is independent of $h$ and $k.$ Now, estimating the fifth
term of the right hand side of the above inequality%
\begin{multline*}
\dsum\limits_{i,j}\left\Vert \left( a_{ij}^{k}-\tau _{h}a_{ij}^{k}\right)
\tau _{h}\partial _{ij}^{2}u_{k}\right\Vert _{L^{2}(\omega _{1})}\leq C_{q,N}%
\underset{i,j}{\max }\sup_{x\in \omega _{1}}\left\vert a_{ij}(x)-\tau
_{h}a_{ij}(x)\right\vert \times \\
\left( \left\Vert \nabla _{X_{2}}^{2}u_{k}\right\Vert _{L^{2}(\omega
_{1}+h)}+\epsilon _{k}^{2}\left\Vert \nabla _{X_{1}}^{2}u_{k}\right\Vert
_{L^{2}(\omega _{1}+h)}+\epsilon _{k}\left\Vert \nabla
_{X_{1}X_{2}}^{2}u_{k}\right\Vert _{L^{2}(\omega _{1}+h)}\right) ,
\end{multline*}%
where $C_{q,N}>0$ is only depends in $q$ and $N$.

Let $\delta >0$ small enough such that for every $\left\vert h\right\vert
\leq \delta $ we have $\omega _{1}+h$ $\subset \subset \Omega $. Then it
follows by \textbf{Corollary 1, }applied on $\omega _{1}+h$, that the
quantity%
\begin{equation*}
\left\Vert \nabla _{X_{2}}^{2}u_{k}\right\Vert _{L^{2}(\omega
_{1}+h)}+\epsilon _{k}^{2}\left\Vert \nabla _{X_{1}}^{2}u_{k}\right\Vert
_{L^{2}(\omega _{1}+h)}+\epsilon _{k}\left\Vert \nabla
_{X_{1}X_{2}}^{2}u_{k}\right\Vert _{L^{2}(\omega _{1}+h)}
\end{equation*}%
is uniformly bounded in $k$ and $h$ (for $\left\vert h\right\vert \leq
\delta $). Since the $a_{ij}$ are uniformly continuous on every $\omega
\subset \subset \Omega $ open then 
\begin{equation*}
\underset{h\rightarrow 0}{\lim }\underset{i,j}{\max }\underset{x\in \omega
_{1}}{\sup }\left\vert a_{ij}(x)-\tau _{h}a_{ij}(x)\right\vert =0,
\end{equation*}%
and hence%
\begin{equation}
\underset{h\rightarrow 0}{\lim }\underset{k\in 
\mathbb{N}
}{\text{ }\sup }\dsum\limits_{i,j}\left\Vert \left( a_{ij}^{k}-\tau
_{h}a_{ij}^{k}\right) \tau _{h}\partial _{ij}^{2}u_{k}\right\Vert
_{L^{2}(\omega _{1})}=0.  \label{17}
\end{equation}

Now, estimating the last term of (\ref{16}). By the triangular inequality we
obtain%
\begin{multline*}
\dsum\limits_{i,j}\left\Vert \partial _{i}a_{ij}^{k}\partial _{j}u_{k}-\tau
_{h}\partial _{i}a_{ij}^{k}\tau _{h}\partial _{j}u_{k}\right\Vert
_{L^{2}(\omega _{1})}\leq \dsum\limits_{i,j}\left\Vert \partial
_{i}a_{ij}^{k}\partial _{j}u_{k}-\tau _{h}\partial _{i}a_{ij}^{k}\partial
_{j}u_{k}\right\Vert _{L^{2}(\omega _{1})} \\
+\dsum\limits_{i,j}\left\Vert \tau _{h}\partial _{i}a_{ij}^{k}\partial
_{j}u_{k}-\partial _{i}\tau _{h}a_{ij}^{k}\tau _{h}\partial
_{j}u_{k}\right\Vert _{L^{2}(\omega _{1})},
\end{multline*}%
and thus, by using the boundedness of the first derivatives of the $a_{ij}$
on $\omega _{1}$ we get 
\begin{multline*}
\dsum\limits_{i,j}\left\Vert \partial _{i}a_{ij}^{k}\partial
_{j}u_{k}-\partial _{i}\tau _{h}a_{ij}^{k}\partial _{j}\tau
_{h}u_{k}\right\Vert _{L^{2}(\omega _{1})} \\
\leq C_{q,N}^{\prime }\underset{i,j}{\max }\underset{x\in \omega _{1}}{\sup }%
\left\vert \partial _{i}a_{ij}(x)-\partial _{i}\tau _{h}a_{ij}(x)\right\vert
\left( \epsilon _{k}\left\Vert \nabla _{X_{1}}u_{k}\right\Vert
_{L^{2}(\omega _{1})}+\left\Vert \nabla _{X_{2}}u_{k}\right\Vert
_{L^{2}(\omega _{1})}\right) \\
+M^{\prime }\left( \epsilon _{k}\left\Vert \nabla
_{X_{1}}U_{k}^{h}\right\Vert _{L^{2}(\omega _{1})}+\left\Vert \nabla
_{X_{2}}U_{k}^{h}\right\Vert _{L^{2}(\omega _{1})}\right) ,
\end{multline*}%
where $M^{\prime }\geq 0$ and $C_{q,N}^{\prime }>0$ are independent of $h$
and $k$. Now, since the $\partial _{i}a_{ij}$ are uniformly continuous
(recall that $A\in C^{1}(\Omega )$) on every $\omega \subset \subset \Omega $
then%
\begin{equation*}
\underset{h\rightarrow 0}{\lim \text{ }}\underset{i,j}{\max }\underset{x\in
\omega _{1}}{\text{ }\sup }\left\vert \partial _{i}a_{ij}(x)-\tau
_{h}\partial _{i}a_{ij}(x)\right\vert =0,
\end{equation*}%
and therefore, from the above inequality we get%
\begin{equation}
\underset{h\rightarrow 0}{\lim }\underset{k\in 
\mathbb{N}
}{\text{ }\sup }\dsum\limits_{i,j}\left\Vert \partial _{i}a_{ij}^{k}\partial
_{j}u_{k}-\partial _{i}\tau _{h}a_{ij}^{k}\partial _{j}\tau
_{h}u_{k}\right\Vert _{L^{2}(\omega _{1})}=0,  \label{18}
\end{equation}%
where we have used (\ref{19}) and \textbf{Lemma\ 2.}

Passing to the limit in (\ref{16}) by using (\ref{17}), (\ref{18}) and (\ref%
{19}) with \textbf{Lemma 2} we deduce%
\begin{equation*}
\underset{h\rightarrow 0}{\lim }\underset{k\in 
\mathbb{N}
}{\text{ }\sup }\left\Vert G_{k}^{h}\right\Vert _{L^{2}(\omega _{1})}=0.
\end{equation*}%
and the proposition follows.
\end{proof}

\begin{corollary}
For every $\omega \subset \subset \Omega $ open we have 
\begin{equation*}
\begin{array}{cc}
\underset{h\rightarrow 0}{\lim }\underset{k\in 
\mathbb{N}
}{\text{ }\sup }\left\Vert \tau _{h}\nabla _{X_{2}}^{2}u_{k}-\nabla
_{X_{2}}^{2}u_{k}\right\Vert _{L^{p}(\omega )}=0,\text{ \ \ \ \ \ \ \ \ \ \
\ } &  \\ 
\underset{h\rightarrow 0}{\lim }\underset{k\in 
\mathbb{N}
}{\text{ }\sup }\left\Vert \epsilon _{k}^{2}(\tau _{h}\nabla
_{X_{1}}^{2}u_{k}-\nabla _{X_{1}}^{2}u_{k})\right\Vert _{L^{p}(\omega )}=0,%
\text{ \ \ \ \ \ } &  \\ 
\underset{h\rightarrow 0}{\lim }\underset{k\in 
\mathbb{N}
}{\text{ }\sup }\left\Vert \epsilon _{k}(\tau _{h}\nabla
_{X_{1}X_{2}}^{2}u_{k}-\nabla _{X_{1}X_{2}}^{2}u_{k})\right\Vert
_{L^{p}(\omega )}=0\text{. } & 
\end{array}%
\end{equation*}
\end{corollary}

\begin{proof}
Similar to proof of \textbf{Corollary 1},where we use the compacity of $\bar{%
\omega}$ and \textbf{Proposition 5}.
\end{proof}

Now, we are able to give the proof of the main theorem. Indeed it is similar
to proof of \textbf{Theorem 2, }where we will use \textbf{Corollary 1 }and 
\textbf{Corollary 2. }Let us prove the convergence%
\begin{equation*}
\epsilon ^{2}\nabla _{X_{1}}^{2}u_{\epsilon }\rightarrow 0\text{ in }%
L_{loc}^{2}(\Omega ).
\end{equation*}

Fix $\omega \subset \subset \Omega $ open, and let $u_{k}\in
W_{0}^{1,2}(\Omega )\cap W_{loc}^{2,2}(\Omega )$ be a sequence of solutions
of (\ref{1}), then it follows from\textbf{\ Corollary 1} and \textbf{2 }that
the subset $\left\{ \epsilon _{k}^{2}\nabla _{X_{1}}^{2}u_{k}\right\} _{k\in 
\mathbb{N}
}$ is relatively compact in $L^{2}(\omega )$ then there exists $v^{\omega
}\in L^{2}(\omega )$ and a subsequence still labeled $(\epsilon
_{k}^{2}\nabla _{X_{1}}^{2}u_{k})$ such that 
\begin{equation*}
\epsilon _{k}^{2}\nabla _{X_{1}}^{2}u_{k}\rightarrow v^{\omega }\text{ in }%
L^{2}(\omega ),
\end{equation*}%
and since $\epsilon _{k}^{2}u_{k}\rightarrow 0$ in $L^{2}(\omega )$ then $%
v^{\omega }=0$ (we used the continuity of $\nabla _{X_{1}}^{2}$on $\mathcal{D%
}^{^{\prime }}(\omega )$). Hence by the diagonal process one can construct a
sequence still labeled $(\epsilon _{k}^{2}\nabla _{X_{1}}^{2}u_{k})$ such
that 
\begin{equation*}
\epsilon _{k}^{2}\nabla _{X_{1}}^{2}u_{k}\rightarrow 0\text{ in }%
L_{loc}^{2}(\Omega ).
\end{equation*}%
To prove the convergence for the whole sequence $(\epsilon ^{2}\nabla
_{X_{1}}^{2}u_{\epsilon })_{0<\epsilon \leq 1}$, we can reason by
contradiction (recall that $L_{loc}^{2}(\Omega )$ equipped with the family
of semi norms $(\left\Vert \cdot \right\Vert _{L^{2}(\omega )})_{\omega
\subset \subset \Omega }$ is a Fr\'{e}chet space), and the proof of the main
theorem is finished.

\subsection{A convergence result for some class of semilinear problem}

In this section we deal with the following semilinear elliptic problem 
\begin{equation}
\left\{ 
\begin{array}{cc}
-\func{div}(A_{\epsilon }\nabla u_{\epsilon })=a(u_{\epsilon })+f &  \\ 
u_{\epsilon }\in W_{0}^{1,2}(\Omega ) & 
\end{array}%
\right. ,  \label{semili}
\end{equation}%
where $a:%
\mathbb{R}
\rightarrow 
\mathbb{R}
$ a continuous nonincreasing real valued function which satisfies the growth
condition%
\begin{equation}
\forall x\in 
\mathbb{R}
:\left\vert a(x)\right\vert \leq c\left( 1+\left\vert x\right\vert \right) ,
\label{crowth}
\end{equation}%
for some $c\geq 0.$ This problem has been treated in \cite{chok} for $f\in
L^{p}(\Omega )$, $1<p\leq 2$, and the author have proved the convergences%
\begin{equation}
\epsilon \nabla _{X_{1}}u_{\epsilon }\rightarrow 0\text{, }u_{\epsilon
}\rightarrow u_{0}\text{, }\nabla _{X_{2}}u_{\epsilon }\rightarrow \nabla
_{X_{2}}u_{0}\text{ in }L^{p}(\Omega ),  \label{conv}
\end{equation}%
where $u_{0}$ is the solution of the limit problem.

Let $f\in L^{2}(\Omega )$ and assume $A$ as in \textbf{Theorem 1} then the
unique $W_{0}^{1,2}(\Omega )$ weak solution $u_{\epsilon }$ to (\ref{semili}%
) belongs to $W_{loc}^{2,2}(\Omega )$. Following the same arguments exposed
in the above subsection one can prove the theorem

\begin{theorem}
Under the above assumptions we have $u_{\epsilon }\rightarrow u_{0}$ in $%
V_{loc}^{2,2}$, $\epsilon ^{2}\nabla _{X_{1}}^{2}u_{\epsilon }\rightarrow 0$
and $\epsilon \nabla _{X_{1}X_{2}}^{2}u_{\epsilon }\rightarrow 0$ strongly
in $L_{loc}^{2}(\Omega ).$
\end{theorem}

\begin{proof}
The arguments are similar, we only give the proof for the Laplacian case, so
assume that $A=Id$.

Let $\omega \subset \subset \Omega $ open$,$ then one can choose $\omega
^{\prime }$ open such that $\omega \subset \subset \omega ^{\prime }\subset
\subset \Omega ,$ let $\rho \in \mathcal{D}(%
\mathbb{R}
^{N})$ with $\rho =1$ on $\omega $, $0\leq \rho \leq 1$ and $Supp(\rho
)\subset \omega ^{\prime }$. Let $0<h<dist(\partial \omega ^{\prime },$ $%
\Omega ),$ we use the same notations of the above subsection, we set $%
U_{k}^{h}=\tau _{h}u_{k}-u_{k}$, then $U_{k}^{h}\in W^{1,2}(\omega ^{\prime
})$ and we have%
\begin{equation*}
-\epsilon _{k}^{2}\Delta _{X_{1}}U_{k}^{h}(x)-\Delta
_{X_{2}}U_{k}^{h}(x)=F^{h}(x)+\tau _{h}a(u)(x)-a(u)(x)\text{, \ a.e }x\in
\omega ^{\prime },
\end{equation*}%
with $F^{h}=\tau _{h}f-f$. We set $\mathcal{W}_{k}^{h}=\rho U_{k}^{h}$ then
we get as in \textbf{Proposition 2} 
\begin{multline*}
\left\Vert \tau _{h}\nabla _{X_{2}}^{2}u_{k}-\nabla
_{X_{2}}^{2}u_{k}\right\Vert _{L^{2}(\omega )}\leq \left\Vert
F^{h}\right\Vert _{L^{2}(\omega ^{\prime })}+M\left\Vert \epsilon _{k}\nabla
_{X_{1}}U_{k}^{h}\right\Vert _{L^{2}(\omega ^{\prime })} \\
+\left\Vert \tau _{h}a(u_{k})-a(u_{k})\right\Vert _{L^{2}(\omega ^{\prime })}
\\
+M\left\Vert \nabla _{X_{2}}U_{k}^{h}\right\Vert _{L^{2}(\omega ^{\prime
})}+M\left\Vert U_{k}^{h}\right\Vert _{L^{2}(\omega ^{\prime })}.
\end{multline*}%
We can prove easily, by using the continuity of the function $a$ and (\ref%
{crowth}), that the Nemytskii operator $a$ maps continuously $L^{2}(\Omega )$
to $L^{2}(\Omega )$. Therefore, the convergence $u_{k}\rightarrow u_{0}$ in $%
L^{2}(\Omega )$ gives $a(u_{k})\rightarrow a(u_{0})$ in $L^{2}(\Omega )$,
and hence \textbf{Lemma 2 }gives%
\begin{equation*}
\underset{h\rightarrow 0}{\lim \text{ }}\underset{k\in 
\mathbb{N}
}{\sup }\left\Vert \tau _{h}a(u_{k})-a(u_{k})\right\Vert _{L^{2}(\omega )}=0,
\end{equation*}%
and finally the convergences (\ref{conv}) give 
\begin{equation*}
\underset{h\rightarrow 0}{\lim \text{ }}\underset{k\in 
\mathbb{N}
}{\sup }\left\Vert \tau _{h}\nabla _{X_{2}}^{2}u_{k}-\nabla
_{X_{2}}^{2}u_{k}\right\Vert _{L^{2}(\omega )}=0.
\end{equation*}%
Similarly, using boundedness of the sequences $(u_{k})$, $(\epsilon
_{k}\nabla _{X_{1}}u_{k})$, $(\nabla _{X_{2}}u_{k})$and $a(u_{k})$ in $%
L^{2}(\Omega ),$ and boundedness of $\rho $ and its derivatives we get 
\begin{equation*}
\left\Vert \nabla _{X_{2}}^{2}u_{k}\right\Vert _{L^{2}(\omega )}\leq
M^{\prime },
\end{equation*}%
and we conclude as in proof of \textbf{Theorem 2.}
\end{proof}

We complete this paper by giving an open question

\begin{problem}
Let $f\in L^{p}(\Omega )$ with $1<p<2$, and consider problem (\ref{1}). In 
\cite{chok} the author have proved the convergence $u_{\epsilon }\rightarrow
u_{0}$ in the Banach space $V^{1,p}$ defined by%
\begin{equation*}
V^{1,p}=\left\{ u\in L^{p}(\Omega )\text{ }\mid \nabla _{X_{2}}u\in
L^{p}(\Omega )\text{ and }u(X_{1},\cdot )\in W_{0}^{1,p}(\Omega _{X_{1}})%
\text{ a.e }X_{1}\in \Omega ^{1}\text{ }\right\} ,
\end{equation*}%
equipped with the norm%
\begin{equation*}
\left\Vert u\right\Vert _{1,p}=\left( \left\Vert u\right\Vert _{L^{p}(\Omega
)}^{p}+\left\Vert \nabla _{X_{2}}u\right\Vert _{L^{p}(\Omega )}^{p}\right) ^{%
\frac{1}{p}}.
\end{equation*}%
Similarly we introduce the Fr\'{e}chet space 
\begin{equation*}
V_{loc}^{2,p}=\left\{ u\in V^{1,p}\text{ }\mid \nabla _{X_{2}}^{2}u\in
L^{p}(\Omega )\right\} ,
\end{equation*}%
equipped with family of norms 
\begin{equation*}
\left\Vert u\right\Vert _{2,p}^{\omega }=\left( \left\Vert u\right\Vert
_{L^{p}(\Omega )}^{p}+\left\Vert \nabla _{X_{2}}u\right\Vert _{L^{p}(\Omega
)}^{p}+\left\Vert \nabla _{X_{2}}^{2}u\right\Vert _{L^{p}(\omega
)}^{p}\right) ^{\frac{1}{p}}\text{, }\omega \subset \subset \Omega \text{
open.}
\end{equation*}%
Can one prove that $u_{\epsilon }\rightarrow u_{0}$ in $V_{loc}^{2,p}$?
\end{problem}

\bigskip

\end{document}